\documentclass[11pt,reqno]{article}
\usepackage{amsfonts,amssymb,amsthm,amsbsy,latexsym,amscd,amsmath,euscript,enumitem,manfnt, marvosym,verbatim,calc,mathrsfs,tensor,textcomp,color,xcolor,tikz,setspace,upgreek,bbm,bm,datetime}
\usepackage[all,cmtip]{xy}
\usepackage{amsgen,graphicx,dsfont,textcomp}
\usepackage[linkcolor=blue,colorlinks=true]{hyperref}
\usepackage[all]{xy}

\allowdisplaybreaks

\newtheorem{thm}{Theorem}[section]

\newtheorem{lem}[thm]{Lemma}

\newtheorem{cor}[thm]{Corollary}

\theoremstyle{definition}

\newcommand {\ZZ}{{\mathds Z}}

\newcommand {\C}{{\mathds C}}

\newcommand {\Z}{{\mathcal Z}}

\newcommand {\Q}{{\mathds Q}}
\newcommand {\R}{{\mathds R}}

\newcommand {\M}{{\mathcal M}}

\newcommand {\CP}{{\mathds P}}

\newcommand {\D}{{\mathcal D}}

\newcommand {\Rr}{{\mathbb{R}}}

\newcommand{\brac}[1]{\left\langle#1\right\rangle}
\newcommand{\w}{\omega}

\def\Jac{\operatorname{Jac}}
\def\arg{\operatorname{arg}}

\def\div{\operatorname{div}}

\def\mod{\operatorname{mod}}

\def\Spec{\operatorname{Spec}}
\def\Im{\operatorname{\frak{Im}}}

\def\reg{\operatorname{reg}}
\def\Hom{\operatorname{Hom}}
\def\log{\operatorname{log}}

\def\Ext{\operatorname{Ext}}

\def\co-dim{\operatorname{co-dim}}
\title{Higher Chow Cycles on Jacobians of Fermat Curves and Hypergeometric Functions}
\author{Subham Sarkar}
\date{}
\begin{document}

\baselineskip=15pt
\maketitle

\allowdisplaybreaks
\begin{abstract}
In this paper we construct certain higher Chow cycles in the $K_1$ of the Jacobian of Fermat curves, generalising a construction of Collino. We further compute the regulator of these elements in terms of special values of hypergeometric functions. Otsubo has computed the regulator of certain elements of $K_0$ and $K_2$ of Fermat varieties and this paper is along the same lines. 

\noindent {\bf AMS Classification: 19F27, 14C30, 14C35,33C20,11G10}

\end{abstract}

\section{Introduction}

Let $\mathscr{V}_{\Q}$ be the category of smooth projective varieties defined over the field rational numbers. For $X \in \mathscr{V}_{\Q}$, one has  the Abel-Jacobi map 
$$AJ_X: CH_{\text{hom}}^{i}(X)\longrightarrow J^i(X)$$
where $J^i(X)=\dfrac{H^{2i-1}(X,\C)^{*}}{F^i+H^{2i-1}(X,\ZZ)}$ is the $i^{th}$ intermediate Jacobian. 

The {\bf regulator map} is a  generalisation of this map. It is a map from a motivic cohomology group  
to a generalised torus, namely the Deligne cohomology. Beilinson \cite{beil} defined the motivic cohomology groups of  $X$ to be 
$$H^{i}_{\M}(X,\Q(n)):=K^{n}_{2n-i}(X)_{\Q},$$
where $K_{2n-i}^{n}(X)_{\Q}$ is the $n^{th}$ graded piece, with respect to the Adams filtration, on the $(2n-i)^{th}$ rational higher $K$-group.  
 Beilinson defined a regulator map from 
$$\reg_{\Q}:H^i_{\M}(X,\Q(n)) \longrightarrow H^i_{\D}(X,\Q(n))$$ 
for $2n-i\geq 1$. He further defined a `real' regulator map to a real vector space, called the `real' Deligne cohomology and  formulated a set of conjectures to explain the relationship between special values of motivic $L$-functions of $X$ and the $\Q$ structure 
induced by the regulator map. A  particular case is Dirichlet's  class number formula for number fields.

Beilinson \cite{beil} proved most of the conjecture in the case of $K_1$ of  a product of modular curves. He showed that there are at least as many elements in the motivic cohomology group as predicted and the determinant of the regulator  map with respect to certain basis is, up to a non-zero rational number, the special value of the $L$-function of $H^2$. The method of proof is to decompose the motive of $X_0(N)$ into motives of modular forms of weight $2$ and then use some classical results to conclude the result. The conjecture has been proved in very few other cases. 

We consider the conjecture in the case of $K_1$ of varieties coming from  Fermat curves. Let $F_N$ denote the Fermat curve 
$$F_N: x^N+y^N-z^N=0.$$
The curves $F_N$ share many nice properties with modular curves. For one, they have some special points, namely the `trivial solutions' to Fermat Last Theorem,  which behave similar to cusps on a modular curve and allow one to construct functions on the Fermat curves. 


In \cite{otsu2}, Otsubo computed the regulator of an element of $H_{\M}^{2}(F_{N},\Q(2))$ which is the second graded piece of $K_{2}(F_{N})_{\Q}$. He expressed the regulator map as a special values of a hypergeometric functions. Also, in \cite{otsu1} he expressed Abel Jacobi image of Ceresa cycle in the Jacobian  $\Jac(F_{N})$ of a Fermat curve as a special value of a hypergeometric functions.  The domain of Abel Jacobi map is the graded piece  $K_0(\Jac( F_{N}))^{(g-1)}_{\Q}$. 

We consider the middle  case of $K_{1}(\Jac(F_{N}))^{(g)}_{\Q}$. We have the following theorem (Theorem \ref{main-thm})

\begin{thm}
Let $X=\Jac(F_{N})$ where $F_{N}$ is the Fermat curve of degree $N$ defined over $\mathds{Q}$. Let $\Z_{QR,P}\in H^{2g-1}_{\M}(X,\mathds{Q}(g))$ 
be the element constructed in  \eqref{mote-ele}. Then one has 
$$\reg_{\mathds{Q}}(\Z_{QR,P})(\Omega^{a,b}_{N})=2\sum_{j=1}^{N}\left[\mathcal F\left(\frac{b}{N},\frac{j}{N},\frac{a}{N}\right)- \mathcal{F}\left(\frac{a}{N},\frac{j}{N},\frac{b}{N}\right)\right]$$
where $a$ and $b$ are positive integers with $a+b<N$,  $\Omega^{a,b}_{N}=\Gamma_{N}\wedge\tilde{\omega}^{a,b}_{N}\in F^{1}(\wedge^{2}H^{1}(F_{N},\C))$ where $\w^{a,b}_N$ is a particular holomorphic form and $\Gamma_N$ is the Poincar\'{e} dual of a particular homology class and
$${\mathcal F} \left(\frac{a}{N},\frac{j}{N},\frac{b}{N}\right)=\dfrac{1}{j}\dfrac{\beta(\frac{a+j}{N},\frac{b}{N})}{\beta(\frac{a}{N},\frac{b}{N})}
\,_3F_{2}\left({{\dfrac{a+j}{N},\dfrac{j}{N},1} \atop {\dfrac{a+b+j}{N},\dfrac{j}{N}+1}};1\right)$$
is a special value of the  hypergeometric function $\,_3F_{2}$.
\end{thm}

A motivic cycle is said to be indecomposable if it is not a product of cycles in other motivic cohomology groups. A cycle is said to be regulator indecomposable if its regulator to Deligne cohomology with real coefficients is non-zero when computed against  $(1,1)$ form which doesn't lie in the $\R$-span of the Hodge classes. This implies that the cycle is indecomposable. We can use our theorem to provide numerical evidence that our cycle is regulator indecomposable.  

\begin{thm}
For $N=11,13,17,19$ and $23$ the  cycle  $\Z_{QR,P}$  in  $H^{2g-1}_{\M}(\Jac(F_{N}),\mathds{Q}(g))$ is indecomposable.
\end{thm}

\newpage
\section{Notation}
\begin{itemize}

\item $\zeta_N$ -- $e^{\frac{2\pi i}{N}}$, a primitive $N^{th}$ root of unity. 
\item  $E_{N}$ -- $\Q(\zeta_{N})$.
\item $\Jac(C)$  -- the Jacobian variety of a smooth projective curve $C$.
\item $G_{N}$ -- $\ZZ/N \ZZ \oplus \ZZ/N\ZZ$.
\item $\brac{a}$ -- the representative of $a \mod N$  in the set $\{1,\dots,N\}$.
\item  $I_{N}$ -- $\{(a,b) \in  G_{N} | a, b, a+b \neq 0 \mod N \} \subset G_{N}$.
\item $\mathscr{V}_{k}$ -- the category of smooth projective varieties defined
over a number field  $k$.  
\item $H^{i}_{\M}(X, \Q(n))$ -- the Motivic cohomology group of $X$.
\item $H^{i}_{\D}(X, A(n))$ -- the Deligne cohomology group with coefficients in a $\ZZ$-module  $A$. 
\item $\reg_{\mathds{R}}$ --- denotes the regulator map from motivic cohomology to Deligne cohomology with 'real' coefficients.
\item $F_{N}$ -- the Fermat curve of degree $N$ given by the equation 
$$x^N+y^N-z^N=0,$$
considered as a variety over $\Q$.
\item $F_{N,E_{N}}$ -- $F_{N,E_{N}}:=F_{N}\times\Spec(E_{N})\in\mathscr{V}_{E_{N}}$
\item $\Im(z)$  -- denotes the imaginary part of a complex number $z$. 
\end{itemize}

\section{Preliminaries} 

Let $X$ be a smooth projective variety of dimension $n$ defined over $\Q$. 

\subsection{Motivic cohomology}
Let $K_{1}(X)$ be the $1^{st}$ higher algebraic $K$-group.  The motivic cohomology groups of $X$ we consider are graded pieces, with respect to the Adam's filtration,  of the group $K_1(X)_{\Q}$, 
$$H^{2n-1}_{\M}(X,\Q(n)):=K^{(n)}_{1}(X)_{\Q},$$
These groups have  the following presentation. Let $Z^{i}(X)$ be  the free abelian group generated by irreducible subvariety of $X$ of codimension $i$ and $Z$ be a  irreducible subvariety  of codimension $n-1$ in $X$. Let $j:\tilde{Z}\rightarrow Z$ be a normalisation of $Z$. We denote $\div_{Z}(f):=j_{*}\div_{\tilde{Z}}(f)\in Z^{n}(X)$.   
Then one has 
\[H^{2n-1}_{\M}(X,\Q(n)):=\dfrac{\left(\underset{Z\in Z^{n-1} (X)}{\bigoplus}k_{Z}^{*}\overset{\oplus\div_{Z}(f)}{\longrightarrow} Z^{n}(X)\right)}{\left(K_{2}(X)\longrightarrow \underset{Z\in Z^{n-1} (X)}{\bigoplus}k_Z^{*}\right)}\otimes\Q,\]
where $k_Z$ is the field of rational functions on $Z$ and $k_Z^{*}$ is the set of all nonzero elements of $k_Z$. 
An element of the above group can be represented by 
$$\Z=\sum_{i=1}^{t}(Z_{i},f_{i})$$
such that $Z_{i}\in Z^{n-1}(X)$ and  $f_{i}\in k(Z_{i})^{*}$  are such that $$\sum_{i=1}^{t}\div_{Z_{i}}(f_{i})=0.$$

\subsection{ Elements in the motivic cohomology group}
\label{regelement}

Let $C$ be a smooth projective curve of genus $g>0$ which satisfies the following condition: There exist two distinct rational points $Q$ and $R \in C$ be  such that 
$Q-R$ is torsion in the Jacobian of $C$. In this section we construct an element in the group $H^{2g-1}_{\M} (X, \Q(g))$, where $X=\Jac(C)$. 

The above condition means there exists a function 
$$f_{QR}:C\rightarrow \CP^{1}$$
such that 
$$\div (f_{QR})=N(Q)-N(R)$$
for some integer $N$. To determine the function precisely we have to choose another distinct point $P\in C$ and add a requirement that $f_{QR}(P)=1$.

Examples of such curves and points are modular curves $X_{0}(N)$ with $Q$ and $R$ being cusps, Fermat curves with $Q$ and $R$ being the `trivial' solutions of Fermat's last theorem  and hyperelliptic curves with $Q$ and $R$ being Weierstrass points.

Let $C_Q$ be the image of the map $C\rightarrow \Jac(C)$ given by $x\rightarrow x-Q$.  Similarly, let $C^{R}$ be the image of the map $C\rightarrow \Jac(C)$ given by $x\rightarrow R-x$. Let  $f_{Q}$ and $f^{R}$ denote the function $f_{QR}$ considered as a function on $C_Q$ and $C^{R}$ respectively.

Consider the cycle 
\begin{equation}\label{mote-ele}
\Z_{QR,P}:=(C_Q,f_{Q}) + (C^{R},f^{R}) .
\end{equation}
As   
$${\div_{C_{Q}}}(f_Q)+\div_{C^{R}}(f^{R})=N(0)-N(R-Q)-N(0)+N(R-Q)=0,$$ 
$\Z_{QR,P}$ lies in  $H^{2g-1}_{\M}(X,\Q(g))$. 

In the case when $C$ is a hyperelliptic curve the above construction  of $\Z\in H^{2g-1}_{\M} (X, \Q(g))$ was done by  Collino \cite{colli}.

\subsubsection{Indecomposable cycles}

Let $L$ be a finite extension of the field $\Q$. Let $X_{L}=X\times_{\Spec(\Q)} \Spec{L}$. The norm map $Nm_{L/\Q}$ induces a map 
$$Nm_{L/\Q}:H^{2g-1}_{\M}(X_{L},\Q(g))\rightarrow H^{2g-1}_{\M}(X,\Q(g)).$$ 
The group of {\em decomposable cycles} is defined to be 
$$H^{2g-1}_{\M}(X,\mathds{Q}(g))_{dec}:=\bigoplus_{[L:\Q]<\infty}Nm_{L/\Q} Im\left(H^{2g-2}_{\M}(X_L,\Q(g-1))\otimes H^{1}_{\M}(X_{L}),\Q(1))\hookrightarrow H^{2g-1}_{\M}(X_{L},\Q(g))\right).$$
The group of  indecomposable cycles is defined to be the quotient
$$H^{2g-1}_{\M}(X,\mathds{Q}(g))_{indec}:=H^{2g-1}_{\M}(X,\mathds{Q}(g))/H^{2g-1}_{\M}(X,\mathds{Q}(g))_{dec}.$$
The indecomposable cycles are the `new' cycles in the motivic cohomology group and it is of interest to determine if a cycle is indecomposable or not. 

\subsection{Deligne Cohomology}

If $X$ is a variety over $\Q$ the $i^{th}$ cohomology group of $X$ is denoted by $H^{i}(X)$. It admits a mixed Hodge structure $(H^{i}(X),W_{\bullet},F^{\bullet})$, where $W_{\bullet}$ is the  weight filtration on $H^{i}(X,\Q)$ and $F^{\bullet}$ is the Hodge filtration on $H^{i}(X,\C)$. The  Deligne cohomology group of $X$ with $\Q$ coefficients is defined to be the hypercohomology of certain  complex, known as Deligne complex.  It turns out that the  Deligne cohomology group can be identified with a group of extensions of mixed Hodge structures. For example, one has 
$$
H^{2g-1}_{\D}(X,\Q(g))=\Ext_{\Q-MHS}^1(\Q(-g),H^{2g-2}(X)).
$$
From the work of Carlson \cite{carl}, this group can be identified  with a certain generalised torus.   In the case of interest to us when $X=\Jac(C)$ from Proposition  $2$ of  \cite{carl} one has the following identification 
\begin{align*}
\Ext_{\Q-MHS}^1(\Q(-g),H^{2g-2}(\Jac(C)))
&\cong \dfrac{H^{2g-2}(\Jac(C),\C)}{F^{g}H^{2g-2}(\Jac(C),\C)+H^{2g-2}(\Jac(C),\Q(g))}
\\
&\cong \dfrac{(F^{1}H^{2}(\Jac(C),\C))^*}{H_{2}(\Jac(C),\Q(2))},
\end{align*}
where  second isomorphism is induced by  Poincar\'{e} duality.

The Deligne cohomology with $\Q$ coefficients is thus  a  generalised torus. In other words it is the $\C$-vector space of linear functionals on the cohomology group $F^{1} H^{2}(\Jac(C))$ modulo the lattice  $H_{2}(\Jac(C),\Q(2))$. The Deligne cohomology with $\R$ coefficients is obtained by extending this map to $\R-MHS$s. In this case it is 
$$H^{2g-1}_{\D}(X,\R(g))=    \dfrac{(F^{1}H^{2}(\Jac(C),\C))^*}{H_{2}(\Jac(C),\R(2))}  \cong      (F^{1}H^{2}(\Jac(C),\R(1)))^*.$$          

\subsection{ Regulator Maps} 
\label{reg-defn}

If $X$ is a smooth projective variety over $\Q$,  Beilinson defined a regulator map from 
$$\reg_{\Q}: H^{2n-1}_{\M}(X,\Q(n)) \longrightarrow H^{2n-1}_{\D}(X,\Q(n)).$$
In the particular case when  $X=\Jac(C)$ and  $n=g$, the motivic cohomology group is $H^{2g-1}_{\M}(X,\Q(g))$. One has the following explicit formula: If
$$\Z=\sum_i(C_i,f_i)$$
be an element of the motivic cohomology group, where $C_i$ and $f_i$ satisfy the conditions above. Let $[0,\infty]$ be the path from $0$ to $\infty$ along the real axis in $\CP^1(\C)$.  Let   $\mu_{i}: \tilde{C_{i}}\rightarrow C_{i}$ be a resolution of singularities. We can think of $f_i$ as  a function on $\tilde{C_i}$. Let 
$$\gamma_i=\mu_{i*}(f_i^{-1}[0,\infty]).$$
From the co-cycle condition we have 
$$\sum_{i=1}^{i=t}\gamma_{i}=\partial(D),$$
where $D$ is $2$-chain in $X$ because $H_{2}(X)$ does not have torsion. The regulator map is defined to be
\begin{equation}\label{exp-1}
 \reg_{\Q} (\Z)(\omega)= \left(\sum_{i=1}^{i=t} \int_{C_i-\gamma_i } \log (f_i) \omega + 2\pi i \int_{D} \omega\right). 
\end{equation}
where $\omega\in F^{1}H^{2}(X,\C)$.

When $C$ is a hyperelliptic curve and $\Z$ is Collino's element, constructed above, Colombo \cite{colo} constructed an extension of mixed Hodge structures coming from the fundamental group of $C$ which corresponds to the regulator of $\Z$. In \cite{sreks}, we generalised her formula to the case when $C$ is a smooth  projective curve of genus $g>0$ and there exist two distinct $\Q$ rational point $\{Q,R\}$ such that $Q-R$ is torsion in $\Jac(C)$. 

\begin{thm}\cite{sreks}\label{ss1}
Let $C$ be a smooth projective curve of genus $g>0$ with the property that there exist two points $Q$, $R$ in $C$ such that $Q-R$ is torsion in $\Jac(C)$. Let us fix a point $P\in C\setminus\{Q,R\}$ and choose a function $f_{QR,P}$ such that $\div(f_{QR,P})=N(Q-R)$ and $f_{QR,P}(P)=1$. Let $\Z_{QR,P}$ be the element of $H^{2g-1}_{\M}(\Jac( C),\Q(g))$ constructed in Section \ref{regelement}. Then there is an extension class 
$\epsilon^{4}_{QR,P} \in \Ext_{MHS}(\Q(-2),\wedge^{2}H^{1}(C))$ constructed from the mixed Hodge structure  on the fundamental groups of $C\setminus \{Q\}$ and $C\setminus \{R\}$ 
such that 
\[ \epsilon^{4}_{QR,P}=(2g+1) \reg_{\Q}(\Z_{QR,P}) \in \Ext^{1}_{\Q-MHS}(\Q(-2),\wedge^{2}H^{1}(C)).\]
\end{thm}
\begin{proof}
This is Theorem $4.18$ in \cite{sreks}. 
\end{proof}
From this theorem we get the following explicit formula for the regulator:

\begin{cor}
Assumptions are as in Theorem \ref{ss1}. Let $\omega$ be a holomorphic $1$-forms on $C$ and $\eta$ another closed $1$-from on $C$. Let $\alpha$ denote the Poincar\'e dual of $\eta$ in $H_1(C)$. Let $\Z_{QR,P}$ be the element of $H^{2g-1}_{\M}(\Jac(C),\Q(g))$. Then one has 
$$ \reg_{\Q}(\Z_{QR,P})(\eta \wedge \omega) =2\int_{\alpha}  \log(f_{QR,P}) \omega.$$
\label{maincor}
\end{cor}
\begin{proof}
This is Corollary 4.19 of \cite{sreks}.
\end{proof}

We can apply this to show the non-triviality of the regulator to Deligne cohomology with Real coefficients. 

\begin{thm}\label{reg-for1}
Assumptions are as in Theorem \ref{ss1}. Let $\w_i$ and  $\w_j$ be two holomorphic 1-forms on $C$. Complex conjugation sends $\w \rightarrow \overline{\w}$. Hence 
$$\Omega_{i,j}= \w_i \wedge  \overline{\w}_j+ \overline{\w}_i \wedge \w_j$$
is invariant under complex conjugation  and so lies in $F^1 \wedge^2 H^1(C,\C) \cap \wedge^2 H^1(C,\R)$.  Let $\gamma_i$ and $\gamma_j$ denote the Poincar\'e duals of $\bar{\w}_i$ and $\bar{\w}_j$ respectively. Then 
$$\reg_{\R}(\Z_{QR,P}) \in H^{2g-1}_{\D}(\Jac(C),\R(g))  \neq 0 \Leftrightarrow \Im \left(\int_{\gamma_i} \log{f_{QR,P}} \w_j-\int_{\gamma_j} \log{f_{QR,P}} \w_j \right)\neq 0$$
\end{thm}

\begin{proof} We have 
$$H^{2g-1}_{\D}(\Jac(C),\R(g)) \cong \dfrac{F^1 H^2(\Jac(C),\C)^*}{H_2(\Jac(C),\R)}.$$
The $2$-form $\Omega_{i,j}$ lies in $F^1 H^2(\Jac(C),\C) \cap H^2(\Jac(C),\R)$.  Since it is in $F^1 H^2(\Jac(C),\C)$, we can consider $ \reg_{\R}(\Z_{QR,P})(\Omega_{i,j})$.  If $\reg_{\R}(\Z_{QR,P})\equiv 0$ in the Deligne cohomology group it lies in the image of $H_2(\Jac(C),\R)$. In that case, 
$$\reg_{\R}(\Z_{QR,P})(\cdot) =\int_{\Gamma} \cdot$$  
for some $2$-chain $\Gamma$ in $H_2(\Jac(C),\R)$. Since  $\Omega_{i,j} \in H^2(\Jac(C),\R)$, 
$$\reg_{\R}(\Z_{QR,P})(\Omega_{i,j}) = \int_{\Gamma} \Omega_{i,j} \in \R$$
So if 
$$\Im  \reg_{\R}(\Z_{QR,P})(\Omega_{i,j}) \neq 0$$
then $ \reg_{\R}(Z_{QR,P})$ is non-trivial. 

From Corollary \ref{maincor} we have 
$$\reg_{\Q}(\Z_{QR,P})(\overline{\w}_i \wedge \w_j)=2 \int_{\gamma_i}\log f_{QR,P} \w_j.$$
and 
$$ \reg_{\Q}(\Z_{QR,P})(\w_i \wedge \overline{\w}_j )=-\langle\reg_{\mathds{Q}}(\Z_{QR,P}),\overline{\w}_j \wedge \w_i \rangle=-2\int_{\gamma_j}\log f_{QR,P} \w_i.$$
Hence we obtain  
\begin{equation}\label{regr}
\Im\left(\reg_{\R}(\Z_{QR,P})(\Omega_{i,j})\right)=2 \Im\left(\int_{\gamma_{i}}\log f_{QR,P} \w_j - \int_{\gamma_{j}} \log f_{QR,P} \w_i\right)
\end{equation}
\end{proof}
\subsection{Hypergeometric functions}

Hypergeometric functions were introduced by Euler and generalised by Gauss, Appell and Thomae. In this  section we discuss certain properties of hypergeometric functions. In a later section we will relate certain values of hypergeometric functions to the regulator image of  some cycles. 

Let $\alpha \in \C\setminus \{0,-1,-2,\dots \}$ and $j$ be a non-negative
integer. Define the {\bf Pochhammer Symbol} to be 
$$(\alpha)_{j}=\alpha(\alpha+1)\cdots (\alpha+j-1)=\frac{\Gamma(\alpha+j)}
{ \Gamma(\alpha)}.$$ 
A generalised hypergeometric function of type $(n,n-1)$ is a function  defined by the series expansions 

$$\,_nF_{n-1}\left({{a_{1}, \dots, a_{n}} \atop {b_{1},\dots, b_{n-1}}};z \right)=\sum_{j \geq 0} \frac{(a_{1})_{j} \dots (a_{n})_{j}}{(b_{ 1 })_{j} \dots (b_{n-1})_{j}(1)_{j}}z^j.$$
It converges absolutely for $|z|<1$ and converges at $z=1$ if $\sum_{j=1}^{j=n-1}b_{j}-\sum_{j=1}^{j=n}a_{j}>0$. It has an analytic continuation to the entire complex plane.  References for this  are \cite{catn},\cite{beuk},\cite{beuk1}. 

We are mainly interested in hypergeometric functions of type $(3,2)$.  A hypergeometric function of type $(3,2)$ has  the following integral representation which allows one to analytically continue it to the entire complex plane $\C$. 

\begin{lem}\label{hyp-lem1}
\[\,_3F_{2}\left({{a_{1},a_{2},a_{3}} \atop {b_{1},b_{2}}};z\right)=\frac{\Gamma(b_{1})\Gamma(b_{2})}{\Gamma(a_{1})\Gamma(b_{1}-a_{1})\Gamma(a_{3})\Gamma(b_{2}-a_{3})}\times\]
\begin{equation}
\displaystyle{\iint_{\Delta}u^{a_{1}-1}(1-zu)^{-a_{2}}(1-u(1-v)^{-1})^{b_{1}-a_{1}-1}v^{b_{2}-a_{3}-1}(1-v)^{a_{3}-a_{1}-1}dudv}, 
\end{equation}
where $\Delta=\{(u,v)|u,v,1-u-v\geq0\}$ and branches of logarithm is chosen as follows- $$\arg(u)=0, \arg(1-u)=0 \text{ and  } \arg(1-zu)\leq\pi/2.$$ 
\end{lem}
\begin{proof}
See $\S4.10$, \cite{otsu2} . 
\end{proof}
\section{Specializing to Fermat Curves} 

\subsection{Fermat curves.}

Let $F_{N}$ be the Fermat curve of degree $N$ over $\Q$ defined by the homogeneous equation 
$$F_N:x^N+y^N=z^N.$$
The set of complex points $F_{N}(\C)$ forms a Riemann surface of genus $\frac{(N-1)(N-2)}{2}$. Let $\zeta_{N}=e^{2\pi i/N}$  be a primitive $N^{th}$ root of unity. The set of points with one of $x$, $y$ or $z$ being $0$ are $(\zeta^{r}_{N}:0:1)$, $(0:\zeta^{r}_{N}:1)$ and $(\zeta^{r}_{2N}:1:0)$, where $r\in\ZZ$. These  are called cusps or the points at infinity. There are $3N$ such points over $E_N=\Q(\zeta_{N})$.

Let  $G_{N}=\ZZ/N\ZZ \oplus \ZZ/N\ZZ $. We think of an element as $g_{N}^{r,s}$ with the group action being given multiplicatively 
$$g_{N}^{r,s} \cdot g_{N}^{t,u}=g_{N}^{r+t,s+u}.$$
The action of 
$G_N$ on $F_{N,E_{N}}:=F_N \times_{\Spec{\Q}} \Spec(E_N)$ is defined by 
\begin{align*} 
G_N \times F_{N,E_N} &\longrightarrow F_{N,E_N}\\
(g_{N}^{r,s},[x:y:z]) &\longrightarrow [\zeta_{N}^{r}x:\zeta_{N}^{s}y:z].
\end{align*}

\subsection{Differentials on Fermat curves.}
\label{differma}

For $a,b \geq1$ define
$$\omega_{N}^{a,b}=x^ay^{b-N}\frac{dx}{x}=-x^{a-N}y^b\frac{dy}{y}.$$
Then $\omega_{N}^{a,b}$ is a  differential form on the manifold $F_{N}(\C)$. The forms $\omega_{N}^{a,b}$ may have poles at the 
points at infinity. 

Let $I_N=\{(a,b) \in \ZZ/N\ZZ \times \ZZ/N\ZZ |a, b, a+b \neq 0\}$. For  $(a,b)\in I_{N}$ we define  $\omega^{a,b}_{N}=\omega^{\brac{a},\brac{b}}_{N}$, where $\brac{a}$ is the representative of $a \mod N$  in the set $\{1,\dots,N\}$. 

\begin{lem} Let $\omega_{N}^{a,b}$ be as above. We have
\label{hol-le}
\begin{enumerate}
\item $\{\omega_{N}^{a,b}: (a,b)\in I_{N}\}$ forms a basis for $H^1(F_{N}(\C),\C)$.
\item $\{\omega_{N}^{a,b}:(a,b)\in I_{N}, a+b < N\}$ forms a basis for  the  subspace of holomorphic forms $H^{1,0}(F_{N})$.
\item Let $a,b \in \ZZ$ and $1 \leq i \leq N$. Then for $j \geq 0$,  
$$\omega^{a+i+jN,b}_{N}=\frac{(\frac{a+i}{N})_j}{(\frac{a+i+b}{N})_j}\text{  }\omega^{a+i,b}_{N}$$
in $H^{1}(F_{N}(\C),\C)$.\label{lem-diff}
\end{enumerate}
\end{lem}

\begin{proof} 
Statements $1$ and  $2$ can be found in  Chapter $3$ of \cite{lang1}, Theorem~$2.1$ and Theorem~$2.2$ respectively. Proof of $3$  is given in Lemma~$4.19$ of \cite{otsu2}.
\end{proof}
Let $\delta_{N}$ be the path in $F_{N}(\C)$ defined by $$\delta_{N}:[0,1]\longrightarrow F_{N}(\C)$$  
$$t\rightarrow (\sqrt[N]{t}, \sqrt[N]{1-t}),$$  
where  the branches are taken in $\R_{\geq0}$. Let $\gamma_{N}$ be the path 
$$\label{h_1path}
\gamma_{N}=\frac{1}{N^2}\sum_{(r,s)\in G_{N}} (1-g_{N}^{r,0})(1-g_{N}^{0,s})\delta_{N}.
$$
$\gamma_{N}$ is a closed path  independent of the choice of $\zeta_{N}$. Hence it represents 
a homology class in $H_{1}(F_{N}(\C),\ZZ)$. 

The pairing 
$$\langle \cdot, \cdot  \rangle:H_{1}(F_{N}(\C),\ZZ) \times
H^{1}(F_{N}(\C),\C) \longrightarrow \C$$
\begin{equation}\label{homduala}
\langle\sigma, \omega \rangle=\int_{\sigma} \omega. 
\end{equation}
induces the de Rham isomorphism
$\Hom(H_{1}(F_{N}(\C),\ZZ),\C) \simeq  H^{1}(F_{N}(\C),\C)$. In particular for  $(a,b)\in I_{N} $ we have
\begin{equation}
\int_{\gamma_{N}}\omega^{a,b}_{N} =\frac{1}{N}\beta(\dfrac{\brac{a}}{N}, \dfrac{\brac{b}}{N}),
\label{betaformula}
\end{equation}
where $\beta(m,n)$ is the Beta function 
$$\beta(m, n)=\dfrac{\Gamma(m)\Gamma(n)}{\Gamma(m+n)}=\int_{0}^{1} u^{m-1}(1-u)^{n-1} du.$$
The group $G_N$ acts on $F_{N}(\C)$, therefore  acts on $H^{1}(F_{N}(\C),\C)$ and we have  $\omega^{a,b}_{N}$ is an eigenform for the action of $G_{N}$, namely
$$ g^{r,s}_N \omega^{a,b}_{N}=\zeta^{ar+bs}_{N}\omega^{a,b}_{N}.$$
We choose the following normalisation 
\begin{equation}\label{normal}
\tilde{\w}^{a,b}_{N}=\frac{\omega^{\brac{a},\brac{b}}_{N}}{\frac{1}{N}\beta(\frac{\brac a}{N}, \frac{\brac b}{N})} 
\end{equation}
and one has  
$$\overline{(\tilde{\w}^{a,b}_{N})}=\tilde{\w}^{-a,-b}_{N}$$

\begin{lem} Let $\tilde{\omega}^{a,b}_{N}$ be the differential form on the Fermat curves $F_{N}$. Then  the Poincar\'e dual of $\tilde{\w}_N^{a,b}$  is $\mu_{-a,-b}P^{-a,-b}_{N} \gamma_{N}$ where  $\mu_{a,b}=\dfrac{N^{2}(1-\zeta_{N}^{a})(1-\zeta^{b}_{N})}{1-\zeta_{N}^{a+b}}$ and $P_N^{a,b}\gamma_N$ is a projector defined below.  
\label{PD}
 \end{lem}
\begin{proof} A form $\eta \in H^{1}(F_{N},\C)$ is the  Poincar\'{e} dual to a  path $\alpha\in H_{1}(F_{N},\C)$ if and only if 
$$\int_{F_{N}} \w \wedge\eta=\int_{\alpha} \w$$
for all $\w \in H^{1}(F_{N},\C)$. 

We consider the  case when  $\w=\tilde{\omega}^{a,b}_{N}$ and $\eta=\tilde{\omega}^{c,d}_{N}$.
Let
$$P^{c,d}_{N}\gamma_{N}:=\frac{1}{N^2} \sum_{(r,s) \in G_N} \zeta^{-(cr+ds)}_{N}(g_{N}^{r,s})\gamma_{N} .$$
One has
$$\int_{g^{r,s}_N \gamma_N} \tilde{\w}_{N}^{a,b} = \int_{\gamma_{N}}  g_N^{r,s}\tilde{\omega}_{N}^{a,b}
=\zeta_N^{ar+bs}\int_{\gamma_{N}}\tilde{\omega}^{a,b}_{N}
=\zeta_N^{ar+bs}$$
Therefore
$$
\int_{P^{c,d}_N\gamma_{N}} \tilde{\omega}^{a,b}_{N}= \int_{\gamma_{N}} (P^{c,d}_N)^*\tilde{\omega}^{a,b}_{N}=\frac{1}{N^2} \sum_{(r,s)\in G_N} \zeta_N^{(a-c)r+(b-d)s}
= \begin{cases} 1 & \text{ if } (c,d)=(a,b) \\
0 &   \text{ otherwise.} \end{cases} $$
Computing the action of $G_N$ on the exterior product along with the calculation in Proposition 4.2 of \cite{otsu1} one has 
$$\int_{F_{N}(\C)}\tilde{\omega}^{a,b}_{N}\wedge\tilde{ \omega}^{c,d}_{N}=\begin{cases} \dfrac{N^{2}(1-\zeta_{N}^{-c})(1-\zeta^{-d}_{N})}{1-\zeta_{N}^{-c-d}}& \text{ if }(c,d)=(-a,-b)\\ 0 & \text{ otherwise} \end{cases}$$
Hence $\dfrac{N^{2}(1-\zeta_{N}^{-c})(1-\zeta^{-d}_{N})}{1-\zeta_{N}^{-c-d}}P^{-c,-d}_{N} \gamma_{N}$ is the Poincar\'{e} dual of $\tilde{\omega}^{c,d}_{N}$

\end{proof}

\section{Main Theorem}

In this section  we obtain an explicit formula for the regulator of the element $\Z_{QR,P}$ in $H^{2g-1}_{\mathcal{M}}(\Jac(F_N),\Q(g))$ when $Q$ and $R$ are particular points at infinity.  Let $Q=(1:0:1)$, $R=(0:1:1)$ be two points at infinity on $F_N(\C)$. Let $P=(\frac{1}{\sqrt[N]{2}}:\frac{1}{\sqrt[N]{2}}:1)$. The function  
$$f_{QR,P}(x:y:z)=\frac{z-x}{z-y}$$ 
satisfies the property that 
$$\div(f_{QR,P})=NQ-NR \text{ and }  f_{QR,P}(P)=1.$$ 
From \eqref{mote-ele} we can use $f_{QR,P}$ to construct $\Z_{QR,P} \in H_{\mathcal{M}}^{2g-1}(\Jac(F_{N}),\mathds{Q}(g))$. 

Let $\tilde{\omega}^{a,b}_N$ with $a+b<N$ be the  holomorphic differential $1$ form on $F_{N}$ defined above.  Let $\gamma_N$  be the homology class defined above.  Let  $\Gamma_{N} \in H^{1}(F_{N})$ be the  Poincar\'e dual of $\gamma_{N}$. Define 
$$\Omega^{a,b}_{N}=\Gamma_N \wedge \tilde{\omega}^{a,b}_{N} \in \displaystyle F^{1}(\wedge^{2}H^{1}(F_{N},\C)).$$ 
From Corollary \ref{maincor} one has
$$
\label{main-exp}
\reg_{\Q}(\Z_{QR,P})(\Omega^{a,b}_{N})=2\displaystyle \int_{\gamma_{N}}\log{\frac{1-x}{1-y}}\tilde{\omega}^{a,b}_{N}.$$
In the next lemma we explicitly evaluate this integral. 

\begin{lem}\label{keylem} Let $\omega^{a,b}_{N}$ be a  holomorphic $1$ form on $F_{N}(\C)$ and $\gamma_{N}$ be the path in $F_{N}(\C)$ as above. Then 
$$\int_{\gamma_{N}}\log(1-x) \omega^{a,b}_{N}=-\sum_{j=1}^{N} \dfrac{ \beta(\dfrac{a+j}{N},\dfrac{b}{N})}{jN} \,_3F_{2}\left({{\dfrac{a+j}{N},\dfrac{j}{N},1} \atop {\dfrac{a+b+j}{N},\dfrac{j}{N}+1}};1\right)$$
$$\int_{\gamma_{N}}\log(1-y) \omega^{a,b}_{N}=-\sum_{j=1}^{N} \dfrac{\beta(\dfrac{b+j}{N},\dfrac{a}{N})}{jN} \,_3F_{2} \left({\dfrac{b+j}{N},\dfrac{j}{N},1} \atop {\dfrac{a+b+j}{N},\dfrac{j}{N}+1};1\right).$$
\end{lem}

\begin{proof}
Recall from Section \ref{h_1path},  $\gamma_{N}$ is the path in $F_{N}(\C)$ 
$$ \gamma_{N}=\frac{1}{N^2}\sum_{(r,s)\in G_{N}} (1-g_{N}^{r,0})(1-g_{N}^{0,s})\delta_{N},$$
where 
\begin{align*}
\delta_{N}:[0,1]\rightarrow &F_{N}(\C)\\
t\rightarrow&(t^{1/N},(1-t)^{1/N}).
\end{align*}

Let $\delta_{N}^{*}(\omega^{a,b}_{N})$ and $\delta_{N}^{*}f$ be the pullbacks of the differential form $\omega^{a,b}_{N}$ and the rational function $f$ respectively. In particular for $f=1-x$ and $g^{r,s}_{N}\in G_{N}$, one has
\begin{align*}
g^{r,s}_{N}\delta_{N}^{*}(1-x)=&1-\zeta^{r}_{N}t^{1/N}\\
g^{r,s}_{N}\delta_{N}^{*}(1-y)=&(1-\zeta^{s}_{N}(1-t)^{1/N})\\
g^{r,s}_{N}\delta_{N}^{*}(\omega^{a,b}_{N})=&\frac{1}{N}\zeta_{N}^{ar+bs}t^{a/N-1}(1-t)^{b/N-1}dt.
\end{align*}
Since, $|\zeta^{r}_{N}t^{1/N}|\leq1$ and $|\zeta^{r}_{N}(1-t)^{1/N}|\leq1$ for $t\in [0,1]$, we have 
\begin{align*}
\log{g^{r,s}_{N}\delta_{N}^{*}(1-x)}=&\log(1-\zeta_{N}^{r}t^{1/N})=-\sum_{j=1}^{\infty} \dfrac{\zeta^{rj}_{N}t^{j/N}}{j}\\
\log{g^{r,s}_{N}\delta_{N}^{*}(1-y)}=&\log(1-\zeta_{N}^{s}(1-t)^{1/N})=-\sum_{j=1}^{\infty} \dfrac{\zeta^{sj}_{N}(1-t)^{j/N}}{j}.
\end{align*}

We first compute the integral over $g_{N}^{r,s} \delta_N$.

\begin{align*}
\int_{g^{r,s}_{N}\delta_{N}}\log{(1-x)}\omega^{a,b}_{N}=& \int_{0}^{1} \log(g^{r,s}_{N}\delta_{N}^{*}(1-x))g_{N}^{r,s}\delta_{N}^{*}(\omega^{a,b}_{N})\\
=&-\int_{0}^{1}\sum_{j=1}^{\infty} \dfrac{\zeta^{(a+j)r+bs}_{N}}{j} \dfrac{1}{N} t^{\frac{a+j}{N}-1}(1-t)^{\frac{b}{N}-1}dt\\
=&-\int_{\delta_{N}}\sum_{j=1}^{\infty} \dfrac{\zeta^{(a+j)r+bs}_{N}}{j} \omega^{a+j,b}_{N}\\
\end{align*}
We can break up the sum over $j$ into a sum 
$$\displaystyle{\sum_{j=1}^{N} \sum_{k=0}^{\infty} \dfrac{\zeta^{(a+j+kN)r+bs}_N}{j+kN} \omega^{a+j+kN,b}_N}.$$
From Lemma \ref{hol-le}  we have 
$$\omega^{a+j+kN,b}_{N}=\dfrac{(\frac{a+j}{N})_k}{(\frac{a+j+b}{N})_k}\,\omega^{a+j,b}_{N}.$$
Further,  
$$\dfrac{(\frac{j}{N})_k}{(\frac{j}{N}+1)_k}=\dfrac{\frac{j}{N}}{(\frac{j}{N}+k)}$$
therefore 
$$ \dfrac{1}{j+kN}=\dfrac{1}{N(\frac{j}{N}+k)}=\dfrac{1}{j} \cdot \dfrac{(\frac{j}{N})_k}{(\frac{j}{N}+1)_k}$$
So the integrand becomes 
$$\displaystyle{\sum_{j=1}^{N} \sum_{k=0}^{\infty} \dfrac{\zeta^{(a+j+kN)r+bs}_N}{j+kN} \omega^{a+j+kN,b}_N= \sum_{j=1}^{N} \dfrac{\zeta_N^{(a+j)r+bs}}{j} \sum_{k=0}^{\infty} \dfrac{(\frac{a+j}{N})_k (\frac{j}{N})_k }{(\frac{a+j+b}{N})_k (\frac{j}{N}+1)_k}\,\omega^{a+j,b}_{N}} $$
The convergence of  inner sum does not depend on the parameter $j$ and can be expressed as a special value of a hypergeometric series 
$$\sum_{k=0}^{\infty} \dfrac{(\frac{a+j}{N})_k (\frac{j}{N})_k }{(\frac{a+j+b}{N})_k (\frac{j}{N}+1)_k}= \,_3F_{2}\left({{\frac{a+j}{N},\frac{j}{N},1} \atop {\frac{a+b+j}{N},\frac{j}{N}+1}};1\right)$$
Using this in the integral  we get 
\begin{equation}
\int_{g^{r,s}_{N}\delta_{N}}\log{(1-x)}\omega^{a,b}_{N}=- \sum_{j=1}^{N}\dfrac{\zeta_{N}^{(a+j)r+bs}}{j} \,_3F_{2}\left({{\frac{a+j}{N},\frac{j}{N},1} \atop {\frac{a+b+j}{N},\frac{j}{N}+1}};1\right) \int_{\delta_{N}} \omega^{a+j,b}_{N}.\label{intdeltaN}
\end{equation}
Note that the hypergeometric term does not depend on $r$ and $s$. We can now compute the integral over $\gamma_N$. 

\begin{align*}
\int_{\gamma_{N}}\log(1-x)\omega^{a,b}_{N}=&\dfrac{1}{N^2}\sum_{(r,s)\in G_{N}}\displaystyle(\int_{\delta_{N}}\log{(1-x)}\omega^{a,b}_{N}
-\int_{g^{r,0}_{N}\delta_{N}}\log{(1-x)}\omega^{a,b}_{N}\\
&+\int_{g^{r,s}_{N}\delta_{N}}\log{(1-x)}\omega^{a,b}_{N}-\int_{g^{0,s}_{N}\delta_{N}}\log{(1-x)}\omega^{a,b}_{N})\\
&=- \dfrac{1}{N^2}\sum_{j=1}^{N}\left(  \dfrac{1}{j} \,_3F_{2}\left({{\frac{a+j}{N},\frac{j}{N},1} \atop {\frac{a+b+j}{N},\frac{j}{N}+1}};1\right)\int_{\delta_N} \omega_N^{a+j,b} \right) \times \\
&\left(\sum_{(r,s)\in G_{N}}\left( 1-\zeta_{N}^{(a+j)r}+  \zeta_{N}^{(a+j)r+bs}- \zeta_{N}^{bs}\right) \right)
\end{align*}
 From the definition of the $\beta$-function and $\delta_N$, we get  
$$\int_{\delta_N} \omega_N^{a+j,b}= \dfrac{1}{N} \beta(\dfrac{a+j}{N},\dfrac{b}{N}).$$
Further, for a fixed $j$, since $a$ and $b$ are non-zero, 
$$\sum_{(r,s)\in G_{N}}(1-\zeta_{N}^{(a+j)r}+  \zeta_{N}^{(a+j)r+bs}- \zeta_{N}^{bs})=N^2,$$
since the sum over non-trivial roots of unity is $0$. So finally this leads to expression 
$$\int_{\gamma_{N}}\log(1-x)\omega^{a,b}_{N}=-\sum_{j=1}^{N} \dfrac{1}{jN} \,_3F_{2}\left({{\dfrac{a+j}{N},\dfrac{j}{N},1} \atop {\dfrac{a+b+j}{N},\dfrac{j}{N}+1}};1\right) \beta(\frac{a+j}{N},\frac{b}{N}).$$
A similar calculation with $y$ in the place of $x$ gives the second expression. 
\end{proof}

\begin{thm}\label{main-thm}
Let $X=\Jac(F_{N})$ where $F_{N}$ is the Fermat curve of degree $N$ defined over $\mathds{Q}$. Let $\Z_{QR,P}\in H^{2g-1}_{\M}(X,\mathds{Q}(g))$ 
be the element constructed in $\S3.1.1$, \eqref{mote-ele}. Then one has 
$$\reg_{\mathds{Q}}(\Z_{QR,P})(\Omega^{a,b}_{N})=2\sum_{j=1}^{N}\left[\mathcal F\left(\frac{b}{N},\frac{j}{N},\frac{a}{N}\right)- \mathcal{F}\left(\frac{a}{N},\frac{j}{N},\frac{b}{N}\right)\right]$$
where $a$ and $b$ are positive integers with $a+b<N$,  $\Omega^{a,b}_{N}=\Gamma_{N}\wedge\tilde{\omega}^{a,b}_{N}\in F^{1}(\wedge^{2}H^{1}(F_{N},\C))$ and
$${\mathcal F} \left(\frac{a}{N},\frac{j}{N},\frac{b}{N}\right)=\dfrac{1}{j}\dfrac{\beta(\frac{a+j}{N},\frac{b}{N})}{\beta(\frac{a}{N},\frac{b}{N})}
\,_3F_{2}\left({{\dfrac{a+j}{N},\dfrac{j}{N},1} \atop {\dfrac{a+b+j}{N},\dfrac{j}{N}+1}};1\right)$$
is a special value of the  hypergeometric function $\,_3F_{2}$.
\end{thm}

\begin{proof}
We know from \eqref{normal} that 
$$\tilde{\w}^{a,b}_{N}=\frac{\omega^{\brac{a},\brac{b}}_{N}}{\frac{1}{N}\beta(\frac{\brac a}{N}, \frac{\brac b}{N})} .$$
From Lemma (\ref{keylem}) we have 
\begin{align}\label{eq-1}
\int_{\gamma_{N}}\log(1-x)\tilde{\omega}^{a,b}_{N}=&-\sum_{j=1}^{N}\dfrac{1}{j}\,_3F_{2}\left({{\dfrac{a+j}{N},\dfrac{j}{N},1} \atop {\dfrac{a+b+j}{N},\dfrac{j}{N}+1,1}}\right)\dfrac{\beta(\frac{a+j}{N},\frac{b}{N})}{\beta(\frac{a}{N},\frac{b}{N})}
\end{align}
Similarly 
\begin{equation}\label{eq-2}
\int_{\gamma_{N}}\log(1-y)\tilde{\omega}^{a,b}_{N}=-\sum_{j=1}^{N}\dfrac{1}{j} \,_3F_{2}\left({{\dfrac{b+j}{N},\dfrac{j}{N},1} \atop {\dfrac{a+b+j}{N},\dfrac{j}{N}+1}}
;1\right)\dfrac{\beta(\frac{b+j}{N},\frac{a}{N})}{\beta(\frac{a}{N},\frac{b}{N})}.
\end{equation} 
The difference of these two expressions gives us our formula 
\end{proof}

\section{Indecomposablilty of the cycles.}

Recall  that a cycle $\Z\in H^{2g-1}_{\M}(X,\Q(g))$ is indecomposable if it is non-zero in the group  
$$H^{2g-1}_{\M}(X,\Q(g))/ H^{2g-1}_{\M}(X,\Q(g))_{dec}$$
where  $H^{2g-1}_{\M}(X,\Q(g))_{dec}$ is the subgroup of decomposable cycles - namely cycles coming from other motivic cohomology groups which map to it. One way to check this is to compute its real regulator and verify that it does not lie in the image of the regulators of decomposable cycles. The real regulators of decomposable cycles are non-zero only when computed against Hodge classes - hence if we can show that the real regulator of our cycle is non-zero when computed against a $(1,1)$-form orthogonal to the Hodge cycles, it would show that $\Z$ is indecomposable. 

In this section we provide numerical evidence in our case for the cycle $\Z_{QR,P}$ for certain $N$. We use a theorem of Aoki \cite{aoki1} which gives an explicit description of the 
Hodge classes in  $\otimes^{2}H^{1}(C)$ and in particular in $\wedge^{2}H^{1}(C)$.  

\begin{thm}
\label{main-thm}
Let $X=\Jac(F_{N})$ where $F_{N}$ is the Fermat curve of degree $N$ defined over $\Q$. Let $\Z_{QR,P}\in H^{2g-1}_{\M}(X,\Q(g))$ 
be the element constructed in  \eqref{mote-ele}. Let 
$$\Omega^{a,b,c,d}_{N}=\tilde{\w}^{a,b}_{N}\wedge\tilde{\w}^{-c,-d}_{N}+\tilde{\w}^{-a,-b}_{N}\wedge\tilde{\w}^{c,d}_{N}$$
with $a+b<N$ and $c+d<N$ be a $(1,1)$ form in $F^1 \wedge^2 H^1(C,\C)$. Then one has 
\[ \Im(\reg_{\mathds{R}}(\Z_{QR,P})(\Omega^{a,b,c,d}_{N}))=2\delta_{a,c}[\dfrac{\mu_{a,b}}{\brac{b-d}}\,_3F_{2}\left({{\dfrac{d+\brac{b-d}}{N},\dfrac{\brac{b-d}}{N},1} \atop {\dfrac{d+c+\brac{b-d}}{N},\dfrac{\brac{b-d}}{N}+1}};1\right)\frac{
\beta(\frac{d+\brac{b-d}}{N},\frac{c}{N})}{\beta(\frac{c}{N},\frac{d}{N})}-\]
\[\dfrac{\mu_{c,d}}{\brac{d-b}}\,_3F_{2}\left({{\dfrac{b+\brac{d-b}}{N},\dfrac{\brac{d-b}}{N},1} \atop {\dfrac{b+a+\brac{d-b}}{N},\dfrac{\brac{d-b}}{N}+1}};1\right)\frac{
\beta(\frac{b+\brac{d-b}}{N},\frac{a}{N})}{\beta(\frac{a}{N},\frac{b}{N})}]+\]
\[2\delta_{b,d}[\dfrac{\mu_{c,d}}{\brac{c-a}}\,_3F_{2}\left({{\dfrac{a+\brac{c-a}}{N},\dfrac{\brac{c-a}}{N},1} \atop {\dfrac{a+b+\brac{c-a}}{N},\dfrac{\brac{c-a}}{N}+1}};1\right)\frac{
\beta(\frac{a+\brac{c-a}}{N},\frac{b}{N})}{\beta(\frac{a}{N},\frac{b}{N})}-\]
\[\dfrac{\mu_{a,b}}{\brac{a-c}}\,_3F_{2}\left({{\dfrac{c+\brac{a-c}}{N},\dfrac{\brac{a-c}}{N},1} \atop {\dfrac{c+d+\brac{a-c}}{N},\dfrac{\brac{a-c}}{N}+1}};1\right)\frac{
\beta(\frac{c+\brac{a-c}}{N},\frac{d}{N})}{\beta(\frac{c}{N},\frac{d}{N})}]\]

where  $\delta_{\star,\star}$ is the Kronecker delta function and   $\brac a $ denotes the integer which represents $a\, \mod\, N$ in  the set $\{1,2,,N\}.$
\end{thm}
\begin{proof}
The form
$$\Omega^{a,b,c,d}_{N}=\tilde{\w}^{a,b}_{N}\wedge\tilde{\w}^{-c,-d}_{N}+\tilde{\w}^{-a,-b}_{N}\wedge\tilde{\w}^{c,d}_{N}.$$
 is invariant under complex conjugation, so it   lies in  $F^{1}(\wedge^{2}H^{1}(F_{N},\C)) \cap \wedge^2 H^1(F_N,\R)$. 
From Lemma \ref{PD} we have that the Poincar\'e dual of $\tilde{\w}^{a,b}_N$ is  $\mu_{-a,-b}P^{-a,-b}_{N} \gamma_{N}$
where  
$$\mu_{-a,-b}=\dfrac{N^{2}(1-\zeta_{N}^{-a})(1-\zeta^{-b}_{N})}{1-\zeta_{N}^{-a-b}}$$
Note that 
$$  \overline{\mu_{a,b}}=-\mu_{a,b}.$$
namely, it is purely imaginary, 
$$\mu_{a,b}\in\Rr(1).$$ 
From the equation  \eqref{regr} in Theorem \ref{reg-for1} and the above remarks, we obtain 
\[\Im(\reg_{\mathds{R}}(\Z_{QR,P})(\Omega^{a,b,c,d}_{N}))=2 \Im\left(\displaystyle\int_{\mu_{a,b} P^{a,b}_{N}\gamma_{N}}\log{\frac{1-x}{1-y}}\tilde{\omega}^{c,d}_{N}-
\displaystyle\int_{\mu_{c,d}P_N^{c,d}\gamma_N} \log{\frac{1-x}{1-y}}\tilde{\omega}^{a,b}_{N}\right).\]
To evaluate this we first evaluate the integral over  $g^{r,s}_{N}\gamma_{N}$.

\begin{align*}
\int_{g^{r,s}_{N}\gamma_{N}}\log(1-x)\omega^{a,b}_{N}=&\dfrac{1}{N^2}\sum_{(l,m)\in G_{N}}\displaystyle(\int_{g^{r,s}_{N}\delta_{N}}\log{(1-x)}\omega^{a,b}_{N}
-\int_{g^{r+l,s}_{N}\delta_{N}}\log{(1-x)}\omega^{a,b}_{N}\\
&+\int_{g^{l+r,m+s}_{N}\delta_{N}}\log{(1-x)}\omega^{a,b}_{N}-\int_{g^{r,m+s}_{N}\delta_{N}}\log{(1-x)}\omega^{a,b}_{N})\\
\end{align*}
Each individual term can be evaluated using \eqref{intdeltaN}.  This gives 

\begin{align*}
=- \dfrac{1}{N^2}  \sum_{j=1}^{N}\left(\dfrac{\zeta_{N}^{(a+j)r+bs}}{j} \,_3F_{2}\left({{\frac{a+j}{N},\frac{j}{N},1} \atop {\frac{a+b+j}{N},\frac{j}{N}+1}};1\right)\int_{\delta_N} \omega_N^{a+j,b}  \times \left(\sum_{(l,m)\in G_{N}}\left( 1-\zeta_{N}^{(a+j)l}+  \zeta_{N}^{(a+j)l+bm}- \zeta_{N}^{bm}\right) \right)\right)\\
\end{align*}
For a fixed $j$ with $a$ and $b$ non-zero, we have 
$$\sum_{(l,m) \in G_{N}}(1 - \zeta_N^{(a+j)l} + \zeta_N^{(a+j)l+bm}-\zeta_N^{bm}+1)=N^2.$$
Further, the $\int_{\delta_N} \w^{a,b}$ can be expressed in terms of the $\beta$ function. So we finally get that the integral is 
$$-\displaystyle\sum_{j=1}^{N}\dfrac{\zeta_{N}^{(a+j)r+bs}}{j}\,_3F_{2}\left({{\dfrac{a+j}{N},\dfrac{j}{N},1} \atop {\dfrac{a+b+j}{N},\dfrac{j}{N}+1}};1\right)\dfrac{1}{N}\beta(\dfrac{a+j}{N},\dfrac{b}{N})$$
Observe that the only dependence on $r$ and $s$ is in the power of $\zeta_N$. We now use this to compute the integral over 
$$P_N^{c,d}\gamma_N=\frac{1}{N^2} \sum_{r,s} \zeta_N^{-cr-ds} g_N^{r,s} \gamma_N.$$ 
\[
\int_{P^{c,d}_{N}\gamma_{N}}\log(1-x)\omega^{a,b}_{N}=-\sum_{j=1}^{N}\left(\displaystyle\dfrac{1}{j}\,_3F_{2}\left({{\dfrac{a+j}{N},\dfrac{j}{N},1} \atop {\dfrac{a+b+j}{N},\dfrac{j}{N}+1}};1\right)\int_{\delta_{N}}
\omega^{a+j,b}_{N} \times \left( \frac{1}{N^2} \sum_{(r,s)\in G_{N}} \zeta_{N}^{(a+j-c)r+(b-d)s} \right)\right)\]
The sum 
$$\frac{1}{N^2}  \sum_{(r,s) \in G_N} \zeta_{N}^{(a+j-c)r+(b-d)s}=\begin{cases} 1 \qquad a+j\equiv c \, \mod \,N \text{ and } b=d \\ 0 \qquad \text{otherwise} \end{cases}$$
Since $1 \leq j \leq N$, only one term survives and $j=\brac{c-a}$ and we get 
\[=-\dfrac{\delta_{b,d}}{\brac{c-a}N}\,_3F_{2}\left({{\dfrac{a+\brac{c-a}}{N},\dfrac{\brac{c-a}}{N},1} \atop {\dfrac{a+b+\brac{c-a}}{N},\dfrac{\brac{c-a}}{N}+1}};1\right)
\beta(\frac{a+\brac{c-a}}{N},\frac{b}{N}).\]
Normalising using $\tilde{\w}_{N}^{a,b}$ instead of $\w_N^{a,b}$ we get 
\[
\int_{P^{c,d}_{N}\gamma_{N}}\log(1-x)\tilde{\omega}^{a,b}_{N}=-\dfrac{\delta_{b,d}}{\brac{c-a}}\,_3F_{2}\left({{\dfrac{a+\brac{c-a}}{N},\dfrac{\brac{c-a}}{N},1} \atop {\dfrac{a+b+\brac{c-a}}{N},\dfrac{\brac{c-a}}{N}+1}};1\right)\frac{
\beta(\frac{a+\brac{c-a}}{N},\frac{b}{N})}{\beta(\frac{a}{N},\frac{b}{N})}.\]
Similarly we obtain 
\[
\int_{P^{c,d}_{N}\gamma_{N}}\log(1-y)\tilde{\omega}^{a,b}_{N}=-\dfrac{\delta_{a,c}}{\brac{d-b}}\,_3F_{2}\left({{\dfrac{b+\brac{d-b}}{N},\dfrac{\brac{d-b}}{N},1} \atop {\dfrac{b+a+\brac{d-b}}{N},\dfrac{\brac{d-b}}{N}+1}};1\right)\frac{
\beta(\frac{b+\brac{d-b}}{N},\frac{a}{N})}{\beta(\frac{a}{N},\frac{b}{N})}.\]
Observe that all these expressions  are real numbers. Combining them with the fact that $
\mu_{a,b}\in\R(1)$ we get that the integral expression is purely imaginary -- so we can drop the $\Im$ and we have 
\[\Im(\reg_{\mathds{R}}(\Z_{QR,P})(\Omega^{a,b,c,d}_{N}))=2\delta_{a,c}[\dfrac{\mu_{a,b}}{\brac{b-d}}\,_3F_{2}\left({{\dfrac{d+\brac{b-d}}{N},\dfrac{\brac{b-d}}{N},1} \atop {\dfrac{d+c+\brac{b-d}}{N},\dfrac{\brac{b-d}}{N}+1}};1\right)\frac{
\beta(\frac{d+\brac{b-d}}{N},\frac{c}{N})}{\beta(\frac{c}{N},\frac{d}{N})}-\]
\[\dfrac{\mu_{c,d}}{\brac{d-b}}\,_3F_{2}\left({{\dfrac{b+\brac{d-b}}{N},\dfrac{\brac{d-b}}{N},1} \atop {\dfrac{b+a+\brac{d-b}}{N},\dfrac{\brac{d-b}}{N}+1}};1\right)\frac{
\beta(\frac{b+\brac{d-b}}{N},\frac{a}{N})}{\beta(\frac{a}{N},\frac{b}{N})}]+\]
\[2\delta_{b,d}[\dfrac{\mu_{c,d}}{\brac{c-a}}\,_3F_{2}\left({{\dfrac{a+\brac{c-a}}{N},\dfrac{\brac{c-a}}{N},1} \atop {\dfrac{a+b+\brac{c-a}}{N},\dfrac{\brac{c-a}}{N}+1}};1\right)\frac{
\beta(\frac{a+\brac{c-a}}{N},\frac{b}{N})}{\beta(\frac{a}{N},\frac{b}{N})}-\]
\[\dfrac{\mu_{a,b}}{\brac{a-c}}\,_3F_{2}\left({{\dfrac{c+\brac{a-c}}{N},\dfrac{\brac{a-c}}{N},1} \atop {\dfrac{c+d+\brac{a-c}}{N},\dfrac{\brac{a-c}}{N}+1}};1\right)\frac{
\beta(\frac{c+\brac{a-c}}{N},\frac{d}{N})}{\beta(\frac{c}{N},\frac{d}{N})}]\]

\end{proof}

\begin{lem}\cite{aoki1}\label{Hodge-cycle-lem} Let $N>3$ be a prime number. If 
$a+b<N$ and $c+d<N$ then  $\Omega^{a,b,c,d}_{N}:=\tilde{\omega}^{a,b}_{N}\wedge\tilde{\omega}^{-c,-d}_{N}$ is a  Hodge cycle if and only if the triples  $\{a,b,N-(a+b)\}$ and $\{c,d,N-(c+d)\}$ are equal upto a permutation. 
\end{lem}
\begin{proof}
In  Theorem 0.3  \cite{aoki1}, Aoki lists the conditions when $\Omega^{a,b,c,d}_N$ is a Hodge cycle. Since $N$ is prime, there is only the condition above.  
\end{proof}

\begin{thm}\label{res-2}
When $N=11,13,17,19, 23$, numerical evidence suggests that the the cycle  $\Z_{QR,P}$ is  indecomposable. 
\end{thm}
\begin{proof}
From Lemma \ref{Hodge-cycle-lem}  $\Omega^{1,i,1,j}_{N}$ is a Hodge cycle if and only if $i+j+1 \equiv 0\, \mod \, N$. In particular, we can choose $j=2i$ for $i\leq N/4$  and compute 
the regulator $\reg_{\R}(\Z_{QR,P})(\Omega^{1,i,2,2i})$ using the  theorem 6.1. Note that for $a$, $b$ non-zero, the numerical values of $\mu_{a,b}$ may have a non-zero real part though we know $\mu_{a,b}\in \R(1)$.   

For integers $i$ and $N$, let 
$$f(i,N)=\dfrac{\Im\left( \reg_{\R}(\Z_{QR,P})(\Omega^{1,i,1,2i}_{N})\right)}{2N^{2}}.$$ 
The non-vanishing of $f(i,N)$ suggests that the cycle $\Z_{QR,P}$ is indecomposable. The following table gives some values  computed using MATHEMATICA. 

\begin{center}
\begin{tabular}{|c|c|c|c|}
\hline
i & N & f(i,N)\\  \hline
2& 13& 0.0753593\\ \hline
2& 17& 0.0591967\\ \hline
3& 17& 0.0419067\\ \hline
4& 17& 0.0306883\\ \hline
3& 19& 0.0382251\\ \hline
4& 19& 0.0285317\\ \hline
3& 23& 0.0323588
\\ \hline
4& 23& 0.0247137\\ \hline
5& 23& 0.0193323\\ \hline
\end{tabular}
\end{center}
\end{proof}
\bibliographystyle{alpha}

\begin{verse}
Subham Sarkar  \\
Indian Statistical Institute, Bangalore\\
$8^{th}$ mile, Mysore Road, Bangalore 560 059, Karnataka, India\\
Email: subham.sarkar13@gmail.com

\end{verse}

\end{document}